\documentclass[11pt]{amsart}
\usepackage{amsthm,amssymb,url,hyperref}
\usepackage{fullpage}

\theoremstyle{plain}
\newtheorem{theorem}{Theorem}[section]

\newtheorem{conjecture}[theorem]{Conjecture}
\newtheorem{corollary}[theorem]{Corollary}

\theoremstyle{definition}

\def\img#1{\mathrm{Im}(#1)}
\def\aut#1{{\mathrm{Aut}(#1)}}
\def\Z{\mathbb Z}
\def\F{\mathbb F}
\def\Q{\mathcal Q}
\def\vv#1#2{\left(\begin{matrix} #1 \\ #2\end{matrix}\right)}
\def\m#1#2#3#4{\left(\begin{smallmatrix} #1 & #2 \\ #3 & #4\end{smallmatrix}\right)}
\def\mm#1#2#3#4{\left(\begin{matrix} #1 & #2 \\ #3 & #4\end{matrix}\right)}

\title{Medial quasigroups of prime square order}

\author{David Stanovsk\'y}

\address{Department of Algebra, Faculty of Mathematics and Physics, Charles University, Prague, Czech Republic}
\email{stanovsk@karlin.mff.cuni.cz}

\begin{document}

\thanks{Research partly supported by the GA\v CR grant 13-01832S}



\begin{abstract}
We prove that, for any prime $p$, there are precisely $2p^4-p^3-p^2-3p-1$ medial quasigroups of order $p^2$, up to isomorphism.
\end{abstract}

\maketitle

\section{Introduction}\label{sec:intro}

\emph{Medial quasigroups}, i.e., quasigroups satisfying the medial law
\begin{displaymath}
    (x*y)*(u*v)=(x*u)*(y*v),
\end{displaymath}
are one of the classical subjects of quasigroup theory. Yet there are very few enumeration results in literature. The aim of the present paper is to extend earlier results of \cite{Kir,SoS,SV}, by enumerating medial quasigroups of prime square order. 

The fundamental tool to study medial quasigroups (and many other classes of quasigroups), is affine representation.
Given an abelian group $G=(G,+)$, automorphisms $\varphi,\psi$ of $G$, and an element $c\in G$, define a new operation $*$ on the set $G$ by
\begin{displaymath}
    x*y = \varphi(x)+\psi(y)+c.
\end{displaymath}
The resulting quasigroup $(G,*)$ is said to be \emph{affine over the group} $G$, and it will be denoted by $\Q(G,+,\varphi,\psi,c)$; the quintuple $(G,+,\varphi,\psi,c)$ is called an \emph{affine form} of $(G,*)$.
The fundamental Toyoda-Bruck theorem \cite[Theorem 3.1]{Sta-latin} states that a quasigroup is medial if and only if there is an abelian group $G=(G,+)$, a pair of \emph{commuting} automorphisms $\varphi,\psi$ of $G$, and $c\in G$ such that $Q=\Q(G,+,\varphi,\psi,c)$. We refer to \cite{Sta-latin} for a detailed account on various kinds of affine representations of quasigroups.

Let $mq(n)$ denote the number of medial quasigroups of order $n$, and $mq(G)$ the number of medial quasigroups that admit an affine form over a group $G$, up to isomorphism. It follows from the classification of finite abelian groups that $mq(G\times H)=mq(G)\cdot mq(H)$ whenever $G,H$ are abelian groups of coprime order. Therefore, the function $mq(n)$ is multiplicative. In particular, if $n=p_1^{k_1}\cdot\ldots\cdot p_m^{k_m}$ is a prime factorization of $n$, then $mq(n)=mq(p_1^{k_1})\cdot\ldots\cdot mq(p_m^{k_m})$. 
Since isotopic groups are isomorphic, a medial quasigroup cannot admit affine forms over two non-isomorphic groups, and thus $mq(n)=\sum mq(G)$ where the sum runs over all isomorphism representatives of abelian groups of order $n$. 
For details, we refer to \cite{SV}.

Quasigroups affine over a cyclic group were enumerated in \cite{Kir,SV}, obtaining an explicit formula
\[mq(\Z_{p^k})=p^{2k} + p^{2k-2} - p^{k-1} - \sum_{i=k-1}^{2k-1} p^i\]
for every prime $p$. In particular, we have 
\begin{align*}
mq(p)=mq(\Z_p)&=p^2-p-1\\
mq(\Z_{p^2})&=p^4-p^3-2p.
\end{align*}
The main result of the present paper is:

\begin{theorem} \label{Th:Zp2} 
$mq(\Z_p^2)=p^4-p^2-p-1$ for every prime $p$.
\end{theorem}

\begin{corollary}
$mq(p^2)=mq(\Z_{p^2})+mq(\Z_p^2)=2p^4-p^3-p^2-3p-1$ for every prime $p$.
\end{corollary}

The proof of Theorem \ref{Th:Zp2} occupies the whole Section 2. The affine forms of the quasigroups are explicitly expressed in Table \ref{tab:main}.
Our formula agrees with the computer calculations of \cite{SV} which presents enumeration of all quasigroups affine over an abelian group of order $<64$ (of order $<128$ with a few exceptions).

An important special case, the \emph{idempotent} medial quasigroups (or \emph{latin affine quandles}, in the quandle terminology \cite{HSV}), has been studied earlier extensively. The enumeration problem is significantly simpler, since the parameters in the affine form can be taken $c=0$ and $\psi=id-\varphi$. Therefore, the enumeration of idempotent medial quasigroups up to isomorphism reduces to the enumeration of fixpoint free automorphisms of abelian groups up to conjugacy (cf. Theorem \ref{Th:Alg}). The strongest results were obtained by Hou \cite{Hou}, providing explicit formulas for orders $p^k$ with $k\leq4$. More information about the idempotent case can be found also in \cite{SiS}. 


\section{Proof of Theorem \ref{Th:Zp2}}

We will follow the enumeration procedure described in detail in \cite{SV}. It is based on the following theorem, originally proposed by Dr\'apal \cite{Dra}.

\begin{theorem}[{{\cite[Theorem 3.2]{Dra}}, \cite[Theorem 2.5]{SV}}]\label{Th:Alg}
Let $G$ be an abelian group. The isomorphism classes of medial quasigroups affine over $G$ are in one-to-one correspondence with the elements of the set
\begin{displaymath}
    \{(\varphi,\psi,c):\varphi\in X,\,\psi\in Y_\varphi,\,c\in G_{\varphi,\psi}\},
\end{displaymath}
where 
\begin{itemize}
	\item $X$ is a set of conjugacy class representatives of the group $\aut{G}$;
	\item $Y_\varphi$ is a set of conjugacy class representatives of the centralizer subgroup $C_\aut{G}(\varphi)$, for every $\varphi\in X$ 
	(here we consider conjugation inside the group $C_\aut{G}(\varphi)$, not conjugation by all elements of $\aut G$);
	\item $G_{\varphi,\psi}$ is a set of orbit representatives of the natural action of $C_\aut{G}(\varphi)\cap C_\aut{G}(\psi)$ on $G/\img{1-\varphi-\psi}$.
\end{itemize}
\end{theorem}

Indeed, a triple $(\varphi,\psi,c)$ corresponds to the quasigroup $\Q(G,\varphi,\psi,c)$, hence, an explicit construction of the sets $X,Y_\varphi,G_{\varphi,\psi}$ provides an explicit construction of the quasigroups.

In the rest of the section, we apply Theorem \ref{Th:Alg} on the group $G=\Z_p^2$. We will identify automorphisms with their matrices, considering $\aut{G}=GL(2,p)$. 
Most of the proof is a bit sketchy and many sentences could have started with the ``it is easy to check that" statement; yet we think that adding more details would not improve readability of the proof.

\begin{table}[ht]
\[
\begin{array}{|l|l|} \hline
\varphi & C(\varphi) \\\hline
\mm a00a,\ a\neq 0 & GL(2,p) \\\hline
\mm a00b,\ 0<a<b & \left\{\mm u00v:\ u,v\neq 0\right\} \\\hline
\mm a10a,\ a\neq 0 & \left\{\mm uv0u:\ u\neq 0\right\} \\\hline
\mm 01ab,\ x^2-bx-a \text{ irreducible} & \left\{\mm uv{av}{u+bv}:\ u\neq0\text{ or }v\neq 0\right\}.\\\hline
\end{array}
\]
\caption{Conjugacy class representatives in $GL(2,p)$ and their centralizer subgroups.}
\label{tab:X}
\end{table}

\begin{proof}[Proof of Theorem \ref{Th:Zp2}]
Let $G=\Z_p^2$. 
The set $X$ of conjugacy class representatives in $\aut{G}=GL(2,p)$ can be chosen as in Table \ref{tab:X}. The four types of representatives correspond to the diagonalizable matrices with one eigenvalue, the diagonalizable matrices with two distinct eigenvalues, the non-diagonalizable matrices with an eigenvalue in $\F_p$, and the non-diagonalizable matrices with eigenvalues in the quadratic extension, respectively. The last case is represented by matrices $\m 01ab$ such that the polynomial $x^2-bx-a$ is irreducible over $\F_p$. 

The centralizer subgroups are also displayed in Table \ref{tab:X} (here and later on, we will omit the index in the centralizer notation).
In the first case, we can take $Y_\varphi=X$ and we have $C(\varphi)\cap C(\psi)=C(\psi)$ for every $\psi\in Y_\varphi$. 
In the remaining three cases, the key observation is that the centralizer subgroups are commutative, hence we can take $Y_\varphi=C(\varphi)$, and we have $C(\varphi)\cap C(\psi)=C(\varphi)$ for every $\psi\in Y_\varphi$. 

The size of $G_{\varphi,\psi}$ will be determined by the following procedure:
If $1-\varphi-\psi$ is a regular matrix, then $|G/\img{1-\varphi-\psi}|=1$, and thus also $|G_{\varphi,\psi}|=1$.
If the rank of the matrix $1-\varphi-\psi$ is one, then $G/\img{1-\varphi-\psi}\simeq\Z_p$, and since all of the centralizer subgroups contain all scalar matrices $\m u00u$, we can always take $G_{\varphi,\psi}=\{\mathbf 0,\mathbf w\}$ where $\mathbf w$ is any non-zero vector.
If the rank of the matrix $1-\varphi-\psi$ is zero, then $G/\img{1-\varphi-\psi}\simeq G$, and the situation depends on $C(\varphi)\cap C(\psi)$, to be discussed below in each particular case.

The results are summarized in Table \ref{tab:main}. Below we give comments on how the table is calculated.

\begin{table}[ht]
\[
\begin{array}{|c|c|l|l|} \hline
\varphi & \psi & c & \text{number} \\\hline
\mm a00a & \mm u00u & \vv00 \quad\text{if } u\neq 1-a & p^2-3p+3 \\\cline{3-4}
a\neq 0 & u\neq 0 & \vv00,\vv10 \quad\text{if } u=1-a & 2(p-2) \\\cline{2-4}
 & \mm u00v & \vv00 \quad\text{if } u,v\neq 1-a & \frac12(p-2)(p^2-4p+5) \\\cline{3-4}
 & 0<u<v& \vv00,\vv10 \quad\text{if } u=1-a \text{ or } v=1-a & 2(p-2)^2 \\\cline{2-4}
 & \mm u10u & \vv00 \quad\text{if } u\neq 1-a & p^2-3p+3 \\\cline{3-4}
 & u\neq 0 & \vv00,\vv10 \quad\text{if } u=1-a & 2(p-2) \\\cline{2-4}
 & \mm 01uv & \vv00 & \frac12 p(p-1)^2 \\
 & x^2-vx-u \text{ irr.} && \\\hline

\mm a00b & \mm u00v & \vv00 \quad\text{if } u\neq 1-a,\ v\neq 1-b & \frac12(p-2)^2(p^2-3p+4) \\\cline{3-4}
0<a<b & u,v\neq 0 & \vv00,\vv10 \quad\text{if } \begin{cases} u=1-a,\ v\neq 1-b \\ u\neq 1-a,\ v=1-b\end{cases} & 2(p-2)(p^2-4p+5) \\\cline{3-4}
 && \vv00,\vv10,\vv01,\vv11 \ \text{if } (u,v)=(1-a,1-b) & 2(p-2)(p-3) \\\hline

\mm a10a & \mm uv0u & \vv00 \quad\text{if } u\neq 1-a & p(p^2-3p+3) \\\cline{3-4}
a\neq0 & u\neq 0 & \vv00,\vv10 \quad\text{if } u=1-a,\ v\neq-1 & 2(p-1)(p-2) \\\cline{3-4}
 && \vv00,\vv10,\vv01 \quad\text{if } u=1-a,\ v=-1 & 3(p-2) \\\hline

\mm 01ab & \mm uv{av}{u+bv} & \vv00 \quad\text{if } (u,v)\neq(1,-1) & \frac12(p^2-p)(p^2-2) \\\cline{3-4}
x^2-bx-a \text{ irr.} & u\neq0\text{ or }v\neq 0 & \vv00,\vv10 \quad\text{if } u=1,\ v=-1 & p^2-p \\\hline 
\end{array}
\]
\caption{Affine forms of medial quasigroups over the group $\Z_p^2$, up to isomorphism.}
\label{tab:main}
\end{table}

\emph{Case $\varphi=\m a00a$.}
Take $Y_\varphi=X$. 

\emph{Subcase $\psi=\m u00u$.} The matrix $1-\varphi-\psi$ is singular iff $u=1-a$. There are $p-2$ such pairs $(\varphi,\psi)$, and since $C(\varphi)\cap C(\psi)=GL(2,p)$, we can choose $G_{\varphi,\psi}=\{\mathbf 0,\mathbf w\}$ with any $\mathbf w\neq 0$. In the remaining $(p-1)^2-(p-2)=p^2-3p+3$ cases, the matrix is regular and $|G_{\varphi,\psi}|=1$.

\emph{Subcase $\psi=\m u00v$.} The matrix $1-\varphi-\psi$ is singular iff $u=1-a$ or $v=1-a$ (we cannot have both at the same time, since $u\neq v$). There are $(p-2)^2$ such pairs $(\varphi,\psi)$, and since the rank of $1-\varphi-\psi$ is one, we have $|G_{\varphi,\psi}|=2$. In the remaining $(p-1)\cdot{p-1\choose 2}-(p-2)^2=\frac12(p-2)(p^2-4p+5)$ cases, the matrix is regular and $|G_{\varphi,\psi}|=1$.

\emph{Subcase $\psi=\m u10u$.} The matrix $1-\varphi-\psi$ is singular iff $u=1-a$. There are $p-2$ such pairs $(\varphi,\psi)$, and since the rank of $1-\varphi-\psi$ is one, we have $|G_{\varphi,\psi}|=2$. In the remaining $(p-1)^2-(p-2)=p^2-3p+3$ cases, the matrix is regular and $|G_{\varphi,\psi}|=1$.

\emph{Subcase $\psi=\m 01uv$.} The matrix $1-\varphi-\psi$ is always regular. Since there are precisely $\frac12(p^2-p)$ irreducible polynomials of degree 2 over $\F_p$, this case contributes $\frac12 p(p-1)^2$ triples $(\varphi,\psi,c)$.

\emph{Case $\varphi=\m a00b$.}
Take $Y_\varphi=C(\varphi)$, the subgroup of diagonal matrices. The total number of pairs $(\varphi,\psi)$ is ${p-1\choose2}(p-1)^2$. For $\psi=\m u00v$, the rank of $1-\varphi-\psi$ is
\begin{itemize}
	\item zero iff $u=1-a$ and $v=1-b$; there are ${p-2\choose 2}$ such pairs $(\varphi,\psi)$, each with $|G_{\varphi,\psi}|=4$, since there are four orbits of the action of $C(\varphi)$ on $\Z_p^2$;
	\item one iff $u=1-a$ or $v\neq 1-b$, or $u\neq 1-a$ and $v=1-b$; for $a=1$, there are $p-2$ choices of $b$, $p-1$ choices of $u$ and one choice of $v$; for $a\neq 1$, there are ${p-2\choose 2}$ choices of $\varphi$, and for each of them $2p-4$ choices of $\psi$; in total, we have $(p-2)(p-1)+{p-2\choose 2}(2p-4)=(p-2)(p^2-4p+5)$ such pairs $(\varphi,\psi)$, each with $|G_{\varphi,\psi}|=2$;
	\item two iff $u\neq 1-a$ and $v\neq 1-b$; these are the remaining pairs $(\varphi,\psi)$, hence, there is ${p-1\choose2}(p-1)^2-{p-2\choose 2}-(p-2)(p^2-4p+5)=\frac12(p-2)^2(p^2-3p+4)$ of them, each with $|G_{\varphi,\psi}|=1$.
\end{itemize}
 
\emph{Case $\varphi=\m a10a$.}
Take $Y_\varphi=C(\varphi)=\{\m uv0u:u\neq0\}$. The total number of pairs $(\varphi,\psi)$ is $p(p-1)^2$. For $\psi=\m uv0u$, the rank of $1-\varphi-\psi$ is
\begin{itemize}
	\item zero iff $u=1-a$ and $v=-1$; there are $p-2$ such pairs $(\varphi,\psi)$, each with $|G_{\varphi,\psi}|=3$, since there are three orbits of the action of $C(\varphi)$ on $\Z_p^2$;
	\item one iff $u=1-a$ or $v\neq -1$; there are $(p-2)(p-1)$ such pairs $(\varphi,\psi)$, each with $|G_{\varphi,\psi}|=2$;
	\item two iff $u\neq 1-a$; these are the remaining pairs $(\varphi,\psi)$, hence, there is $p(p-1)^2-(p-2)-(p-2)(p-1)=p(p^2-3p+3)$ of them, each with $|G_{\varphi,\psi}|=1$.
\end{itemize}
 
\emph{Case $\varphi=\m 01ab$.}
Take $Y_\varphi=C(\varphi)$. For $\psi=\m uv{av}{u+bv}\in C(\varphi)$, the determinant of the matrix $1-\varphi-\psi$ is $(1-u)^2-b(1-u)(1+v)-a(1+v)^2$. Assume the determinant is 0. Then either $1+v=0$, and thus also $1-u=0$, or we can divide by $(1+v)^2$ and obtain the equation $(\frac{1-u}{1-v})^2-b\frac{1-u}{1+v}-a=0$, which has no solution, because the polynomial $x^2-bx-a$ is irreducible. Therefore, the matrix $1-\varphi-\psi$ is singular if and only $u=1$ and $v=-1$. Since $C(\varphi)$ acts transitively on $\Z_p^2-\{\mathbf 0\}$, we have $|G_{\varphi,\psi}|=2$. There are $\frac12(p^2-p)$ irreducible polynomials of degree 2, thus the singular case contributes $p^2-p$ triples. The regular case contributes $\frac12(p^2-p)(p^2-2)$ triples.

Summing up all the contributions (see the last column of Table \ref{tab:main}), we obtain that the total number is $p^4-p^2-p-1$. 
\end{proof}

\section{Concluding remarks}

In \cite[Problem 3.4]{SV}, we asked to calculate $mq(\Z_p^k)$ for any $p,k$.
In theory, using Macdonald's classification of conjugacy classes in general linear groups \cite{Mac}, one could continue in the fashion of Section 2 to higher dimensions. But the complexity of such calculations would grow rapidly.
As an alternative, we propose the following idea.

\begin{conjecture} Let $k$ be any natural number.
\begin{enumerate}
	\item There is an integer polynomial $f_k$ of degree $2k$ such that $mq(\Z_p^k)=f_k(p)$ for every prime $p$.
	\item There is an integer polynomial $g_k$ of degree $2k$ such that $mq(p^k)=g_k(p)$ for every prime $p$.
\end{enumerate}
\end{conjecture}

If the conjecture was true, one could interpolate the polynomials from the values of $mq(\Z_p^k)$ and $mq(p^k)$ for the first $2k+1$ primes.

\end{document}